\documentclass[11pt]{article}

\usepackage{amsmath, amsfonts, amsthm, bbm}
\usepackage{bbm,verbatim,color}

\newcommand{\N}{\mathbb{N}}

\newcommand{\dd}{\mathrm{d}}
\newcommand{\E}{\mathbf{E}}
\newcommand{\p}{\mathbf{P}}

\theoremstyle{plain}
\newtheorem{lemma}{Lemma}

\newtheorem*{theorem*}{Theorem}

\newtheorem*{corollary*}{Corollary}

\theoremstyle{remark}
\newtheorem*{remark*}{Remark}

\title{Branching Processes with Immigration in a Random Environment
-- the Grincevi\v{c}ius--Grey setup}

\author{P\'eter Kevei\thanks{
Bolyai Institute, University of Szeged,
Aradi v\'ertan\'uk tere 1, 6720 Szeged, Hungary;
e-mail: \texttt{kevei@math.u-szeged.hu}} 
}

\date{}

\begin{document}

\maketitle

\begin{abstract}
We determine the tail asymptotics of the stationary distribution
of a branching process with immigration in a random environment, 
when the immigration distribution dominates the offspring distribution.
The assumptions are the same as in the Grincevi\v{c}ius--Grey 
theorem for the stochastic recurrence equation.

\end{abstract}

\section{Introduction and the main result}

A branching process with immigration in a random environment is a usual 
Galton--Watson process with immigration, where the offspring and immigration
distribution in each generation is governed by an independent identically 
distributed (iid) sequence of probability measures.
Let $\Delta$ denote the set of probability measures on $\N = \{0,1,\ldots \}$,
and consider the Borel-$\sigma$-algebra on it induced by the total variation 
distance. Let $\xi, \xi_0, \xi_1, \ldots $ be an iid sequence in $\Delta^2$, the 
components $\xi_n = ( \nu_{\xi_n}, \nu_{\xi_n}^\circ)$ 
represent the offspring and immigration distribution in the consecutive 
generations. Let $X_0 = x \in \N$, and
\begin{equation} \label{eq:def-X}
X_{n+1} = \sum_{i=1}^{X_n} A_i^{(n+1)} + B_{n+1}
=: \theta_{n+1} \circ X_n + B_{n+1}, \quad n \geq 0, 
\end{equation}
where conditioned on the environment 
$\mathcal{E} = \sigma(\xi_0, \xi_1, \ldots)$, the variables 
$\{ A_i^{(n)}$, 
$B_n :  n \in \N, i\geq 1 \}$ are  independent, and
for $n$ fixed, 
$(A_i^{(n)})_{i \geq 1}$ are iid with distribution $\nu_{\xi_n}$, and $B_n$ has 
distribution $\nu^\circ_{\xi_n}$. 
Note that given the environment the random variables are independent however,
$\nu_\xi$ and $\nu^\circ_\xi$ may depend. 
The variable $A_i^{(n)}$ is the number of offsprings 
of the $i$th element in the $(n-1)$st generation, 
and $B_n$ is the number of immigrants in the $n$th generation. 
We write 
$\theta_{n} \circ x  = \sum_{i=1}^{x} A_i^{(n)}$,
and 
$\theta_n\circ (x+y) = \theta_{n} \circ x + \theta_{n} \circ y$, 
keeping in mind that the 
two random operators on the right-hand side are not the same,
but independent conditioned on the environment.

For a random measure $\xi$ denote
\[ 
m(\xi) = \sum_{i=1}^\infty i  \nu_{\xi} (\{i\}) = \E [ A | \xi],
\] 
the expectation of its offspring 
distribution.
Our standing assumption is that for some $\kappa > 0$
\begin{equation} \label{eq:grgr}
\E (m(\xi)^\kappa) < 1,  \quad  
\E (A^{(1 \vee \kappa) + \delta}) < \infty,
\ \text{for some } \delta > 0, 
\end{equation}
where $a \vee b = \max \{ a, b \}$.
Then, by the Jensen inequality the process is subcritical, i.e.
$\E ( \log m(\xi) ) < 0$, which implies that the corresponding 
branching process in a random  environment without immigration 
dies out almost surely, see e.g.~the recent monograph by 
Kersting and Vatutin \cite[Section 2.2]{KerVat}.

By Theorem 3.3 in Key \cite{Key}
(see also Lemmas 1 and 3 in Basrak and Kevei \cite{BK}), under condition 
\eqref{eq:grgr} the Markov chain 
\eqref{eq:def-X}  has a unique stationary distribution $X_\infty$, which
using backward iteration can be represented as
\begin{equation*} 
X_\infty 
\stackrel{\mathcal{D}}{=}
\sum_{i=0}^\infty \theta_0 \circ \theta_1 \circ \ldots \circ \theta_{i-1}  \circ B_{i},
\end{equation*}
where $\stackrel{\mathcal{D}}{=}$ stands for equality in distribution.
The stationary distribution $X_\infty$ satisfies the distributional fixed point 
equation
\begin{equation*} 
X \stackrel{\mathcal{D}}{=} 
\sum_{i=1}^X A_i + B 
\end{equation*}
where  $(\xi,B, A_1,A_2,\ldots)$ and $X$ on the right-hand side are independent,
and conditionally on $\xi$ the variables $B, A_1, \ldots$ are independent, and
$A_1, A_2, \ldots$ are iid $\nu_{\xi}$, and $B$ has distribution $\nu_{\xi}^\circ$.

The main result of the paper characterizes the regular 
variation of the stationary distribution.

\begin{theorem*} 
Assume that there is a $\kappa > 0$ such that \eqref{eq:grgr} holds.
Let $\ell$ be a slowly varying function.
Then
\begin{equation} \label{eq:Btail}
\p ( B > x) \sim \frac{\ell(x)}{x^{\kappa}}, \quad \text{as } \ x \to \infty,
\end{equation}
if and only if
\begin{equation} \label{eq:stat-tail}
\p (X_\infty > x) \sim 
\frac{\ell(x)}{x^{\kappa}} \frac{1}{1 - \E (m(\xi)^\kappa)}, \quad 
\text{as } \ x \to \infty.
\end{equation}
\end{theorem*}

The implication \eqref{eq:Btail} $\Rightarrow$ \eqref{eq:stat-tail} 
was investigated in deterministic 
environment, in which case $m(\xi ) \equiv \E (A) \in (0,1)$ is deterministic. 
Then we obtain a generalization of Theorem 2.1.1 in Basrak et al.~\cite{BKP}, 
where it is assumed that $\kappa \in (0,2)$, and 
if $\kappa \geq 1$ then $\E (A^2) < \infty$. 
(In fact for $\kappa < 1$ we need that $\E (A^{1 + \delta}) < \infty$ for 
some $\delta > 0$, whereas in \cite{BKP} only $\E (A) < 1$ is needed.)
Theorems 2.1 and 2.2 by 
Foss and Miyazawa \cite{FossMiy} covers more 
general tail behavior, not only regular variation. If $\E (B) < \infty$ then
their Theorem 2.1 implies our result, while if $\E (B) = \infty$ it is assumed
in \cite[Theorem 2.2]{FossMiy} that either $\E ( A^2) < \infty$, or 
$1 - \int_{x}^\infty \p ( A > u) \dd u$ is subexponential.
In the latter case, we only require that $\E (A^{1 + \delta}) < \infty$ for 
some $\delta$, which does not seem to imply the subexponential condition
(on a related question, on the subexponential property of the integrated 
tail distribution, see Section 1.4 in Embrechts et al.~\cite{EKM}).
Therefore, in some cases our results are new even in the deterministic 
setup. To the best of our knowledge, the converse implication
\eqref{eq:stat-tail} $\Rightarrow$ \eqref{eq:Btail}  was not treated earlier.

In the random environment setup 
the tail behavior was studied in \cite{BK} under a 
different assumption. In \cite{BK}, 
the main assumptions are Cram\'er's condition 
\begin{equation} \label{eq:cramer}
\E ( m(\xi)^{\kappa} ) = 1, \quad \text{for some } \ \kappa > 0,
\end{equation}
and $\E (A^\kappa) < \infty$, $\E (B^{\kappa}) < \infty$.
If some additional weak assumptions are satisfied, then the tail 
of $X_\infty$ is regularly varying with index $-\kappa$, that is
$\p ( X_\infty > x) \in \mathcal{RV}_{-\kappa}$. Therefore, in
\cite{BK} the tail behavior is governed by the offspring distribution,
while in the present paper the
tail is determined by the immigration.

The main idea of the proof is very simple. If a 
random sum $\sum_{i=1}^B A_i$ is large, where the summands are small,
then, by the law of large numbers, it is asymptotically 
$B \E (A)$, see \cite[Remark 2.4]{FossMiy}.
More precisely, we rely on the similarity between the asymptotic behavior 
of the branching process $(X_n)$ in \eqref{eq:def-X} and the 
stochastic recurrence equation defined by $Y_0 = y \geq 0$,
\begin{equation}  \label{eq:sre}
Y_{n+1} = C_{n+1} Y_n + D_{n+1}, \quad n \geq 0,
\end{equation}
where $(C, D), (C_1, D_1), \ldots$ are iid random vectors, with 
nonnegative components. For recent monographs on the stochastic 
recurrence equation we refer to Buraczewski et al.~\cite{BDM}, and 
to Iksanov \cite{Iksanov}. The tail behavior of the stationary
distribution, or, which is the same the solution to the corresponding 
stochastic fixed point equation,
is well-understood, see e.g.~\cite[Section 2.4]{BDM}: 
If $\E (C^\kappa ) = 1$, and $\E (D^\kappa) < \infty$ then by the 
Kesten--Grincevi\v{c}ius--Goldie theorem 
(called Kesten--Goldie theorem in \cite{BDM})
the tail is asymptotically $k x^{-\kappa}$, for some $k > 0$. 
For an extension of this result
see Kevei \cite{Kevei}.
While, if $\E (C^\kappa) < 1$ and $\E (C^{\kappa + \delta} ) < \infty$
for some $\delta > 0$,
then by the Grincevi\v{c}ius--Grey theorem 
$\p( D > \cdot ) \in \mathcal{RV}_{-\kappa}$ if and only if
$\p ( Y_\infty > \cdot ) \in \mathcal{RV}_{-\kappa}$, where $Y_\infty$
is the stationary distribution of \eqref{eq:sre}.

\smallskip 

The regular variation of the stationary distribution of a Markov chain
has a lot of consequences. Then, the
well-known theory of regularly varying time series apply, see 
e.g.~the recent monograph by Kulik and Soulier \cite{KS19}.
The conditions to apply the general theory, that is ergodicity, anticlustering,
and vanishing small values were all established in \cite{BK}.
In particular, if \eqref{eq:grgr} and \eqref{eq:stat-tail} hold then
all the results in Section 3 in \cite{BK} hold true
in the current setup, because in the proofs only the regular variation
of the stationary distribution was used, and that 
$\E ( m(\xi)^\alpha) < 1$, $\E (B^\alpha) < \infty$, for some $\alpha > 0$.
In particular, the tail process of $(X_n)$, the convergence of the point process, 
and the central limit theorems are given in Theorems 3--5 in \cite{BK}.

\section{Proof} \label{sec:tail}

For the analysis of the stationary distribution we need 
the tail behavior of randomly stopped sums of identically 
distributed, conditionally independent summands, where 
the number of terms dominates the summands. Such results 
for independent summands were subject of intensive investigations, 
see e.g.~Barczy et al.~\cite[Proposition D.3]{BBP},
Ale\v{s}kevi\v{c}ien\.{e} et al.~\cite[Theorem 1.2]{Alev},
Fa\"{y} et al.~\cite[Proposition 4.3]{Fay}, 
Robert and Segers \cite{RobertSegers}.
In the iid case the next statement coincides with 
Proposition D.3 in \cite{BBP} for $\kappa \geq 1$, while for 
$\kappa < 1$ we need $\E ( A^{1 + \delta} ) < \infty$ for some 
$\delta >0$, whereas in \cite{BBP} only $\E ( A )  < 1$ is assumed.

\begin{lemma} \label{lemma:rsum-tail}
Let $A, A_1, \ldots$ be identically distributed random variables,
independent given the environment $\xi$. Let $B$ be an integer-valued
random variable, independent of the $A$'s and $\xi$. Assume that 
$\p ( B > x ) = \frac{\ell(x)}{x^\kappa}$, 
$\E (A^{(1 \vee \kappa) + \delta }) < \infty$,
for some $\kappa > 0$, $\delta > 0$, and a slowly varying function $\ell$.
Then, as $x \to \infty$
\[
\p \left( \sum_{i=1}^B A_i > x \right) \sim 
\p ( B > x) \, \E ( m(\xi)^\kappa ).
\]
\end{lemma}

\begin{proof}
To ease notation write $\widetilde A_i  = A_i - m(\xi)$. 

First we prove that as $x \to \infty$ for some $\alpha > \kappa$
\begin{equation} \label{eq:breiman-aux1}
\p \left(  \sum_{i=1}^B \widetilde A_i > x \right)
= o(x^{-\alpha}).
\end{equation}
In what follows, nonspecified limits are meant as $x \to \infty$,
and $c$ is finite positive constant, whose value is irrelevant, and 
may change from line to line.

By the assumption on $B$, for any $\beta \in (0,\kappa)$ we have 
$\E (B^\beta) < \infty$. Assume that $\kappa > 1$. Then by 
\cite[Lemma 2 (ii)]{BK},
for $1 < \alpha < \kappa +  \delta < 2 \kappa$ 
there exists $c = c(\alpha) $ depending 
only on $\alpha$, such that for any $n \geq 1$
\[
\E \left( \left|  \sum_{i=1}^n \widetilde A_i \right|^\alpha \right) 
\leq 
c n^{1 \vee \frac{\alpha}{2}} \,
\E \left(  \E [ | \widetilde A|^\alpha | \xi ] \right)
\leq c n^{1 \vee \frac{\alpha}{2}} \,
\E ( A^\alpha ).
\]
Thus
\begin{equation} \label{eq:rsum-moment1}
\E \left( \left|  \sum_{i=1}^B \widetilde A_i \right|^\alpha \right) 
\leq 
c \, \E ( B^{1 \vee \frac{\alpha}{2}} ) \, \E ( A^\alpha ) < \infty.
\end{equation}

For $\kappa \leq 1$, choose $\alpha$ and $\eta$ such that 
$\alpha - \eta < \kappa < \alpha < \kappa + \delta$, 
$2 \eta \leq \alpha < 1 + \eta$, and 
$\tfrac{\alpha}{\alpha - \eta} < 1 + \delta$. This is clearly possible,
if $\alpha$ is close enough to $\kappa$ and $\eta$ is small. Then,
by \cite[Lemma 2 (i)]{BK}
\[
\E \left( \left|  \sum_{i=1}^n \widetilde A_i \right|^\alpha \right) 
\leq 
c n^{\alpha - \eta} \,
\E \left(  \E [ | \widetilde A|^{\frac{\alpha}{\alpha - \eta}} | \xi ] \right)
\leq c n^{\alpha - \eta} \, 
\E (  A^{1 + \delta} ),
\]
implying that 
\begin{equation} \label{eq:rsum-moment2}
\E \left( \left|  \sum_{i=1}^B \widetilde A_i \right|^\alpha \right) 
\leq 
c \, \E ( B^{\alpha - \eta } ) \, \E ( A^{1 + \delta} ) < \infty.
\end{equation}
In both cases \eqref{eq:breiman-aux1} follows.

Next
\begin{equation} \label{eq:breiman-upper}
\begin{split}
\p \left( \sum_{i=1}^B A_i > x \right) 
& = \p \left( \sum_{i=1}^B \widetilde A_i + B m(\xi) > x \right) \\
& \leq \p  \left( \sum_{i=1}^B \widetilde A_i > \varepsilon x \right)
 + \p ( B m(\xi) > ( 1 - \varepsilon) x) \\
& \sim \E ( m(\xi)^{\kappa} ) \, ( 1 - \varepsilon)^{-\kappa} \,
\p ( B >  x ),
\end{split}
\end{equation}
by \eqref{eq:breiman-aux1}, 
a version of Breiman's lemma (see e.g.~Embrechts and Goldie 
\cite[Corollary]{EmbGold}) and the regular variation of $\p ( B > x)$.
Breiman's lemma is indeed applicable, as 
$\E (m(\xi)^{(1 \vee \kappa) + \delta} ) \leq 
\E (A^{(1 \vee \kappa) + \delta}) < \infty$. Similarly, for the lower bound
\begin{equation} \label{eq:breiman-lower}
\begin{split}
\p \left( \sum_{i=1}^B A_i > x \right) 
& \geq \p \left( B m(\xi) > ( 1 + \varepsilon) x, 
\sum_{i=1}^B \widetilde A_i > -\varepsilon x \right) \\
& \geq \p ( B m(\xi) > ( 1 + \varepsilon) x) -
\p \left( \left| \sum_{i=1}^B \widetilde A_i \right| > \varepsilon x \right) \\
& \sim \E ( m(\xi)^{\kappa} ) \, ( 1 + \varepsilon)^{-\kappa} \,
\p ( B >  x ).
\end{split}
\end{equation}
Letting $\varepsilon \to 0$ in \eqref{eq:breiman-upper} and 
\eqref{eq:breiman-lower}, the result follows.
\end{proof}

Iterating the previous statement, we obtain the following.

\begin{corollary*} 
For any $i \geq 0$ as $x \to \infty$
\begin{equation} \label{eq:cor}
\p ( \theta_0 \circ \ldots \circ \theta_{i-1} \circ B_i > x) 
\sim \p ( B > x) (\E ( m(\xi)^\kappa) )^{i}.
\end{equation}
\end{corollary*}

For the implication \eqref{eq:stat-tail} $\Rightarrow$ \eqref{eq:Btail}
we need the random sum 
version of Lemma 2 in Grey \cite{Grey}.

\begin{lemma} \label{lemma:grey}
Let $N$ be a nonnegative integer valued random variable for which
$\p ( N > x) \sim C \ell(x) x^{-\kappa}$, with some $C > 0$, and 
a slowly varying function $\ell$. Let $(\xi, B, A_1, A_2, \ldots)$
be independent of $N$, and conditionally on $\xi$ the variables
$(B, A_1, A_2, \ldots)$ are independent, $A_1, A_2, \ldots$ are iid
$\nu_{\xi}$, and $B$ has distribution $\nu_{\xi}^\circ$. Assume 
condition \eqref{eq:grgr}. Then as $x \to \infty$
\[
\p ( B > x ) \sim \frac{\ell(x)}{x^{\kappa}} \ \Longleftrightarrow \ 
\p \left( B + \sum_{i=1}^N A_i > x \right) 
\sim ( 1 + C \, \E (m(\xi)^\kappa) ) \frac{\ell(x)}{x^{\kappa}}.
\]
\end{lemma}

\begin{proof}
The statement follows from Lemma 2 in \cite{Grey},
since the right-hand side above is equivalent to
\[
\p ( B + N m(\xi) > x) \sim ( 1 + C \, \E (m(\xi)^\kappa) ) \ell(x) x^{-\kappa}.
\]
Indeed, 
\[
B + \sum_{i=1}^N A_i = B + N m(\xi) + \sum_{i=1}^N \widetilde A_i, 
\]
where the sum on the right-hand side,
by \eqref{eq:rsum-moment1} or \eqref{eq:rsum-moment2} 
has finite moment of order $\alpha > \kappa$,
implying that its tail is $o(x^{-\alpha})$.
\end{proof}

We are ready to prove the main result.

\begin{proof}[Proof of the Theorem]
As in \cite{Grey}, 
implication \eqref{eq:stat-tail} $\Rightarrow$ \eqref{eq:Btail}
follows from Lemma \ref{lemma:grey} with
the choice $N = X_\infty$ and $C = 1/(1 - \E (m(\xi)^\kappa))$.

We turn to \eqref{eq:Btail} $\Rightarrow$ \eqref{eq:stat-tail}.
We follow the proof of the Grincevi\v{c}ius--Grey theorem 
in \cite[Sect.~2.4.3]{BDM}.

To ease notation write 
$\theta_0 \circ \ldots \circ \theta_{i-1} = \Theta_{i-1}$, and
$m(\xi_0) \ldots m(\xi_{i-1}) = \Pi_{i-1}$, $i = 0, 1,2,\ldots$,
with the convention $\Theta_{-1} \circ B_0 = B_0$ and 
$\Pi_{-1} = 1$.
For $K > 1$ we use the decomposition
\[
X_\infty \stackrel{\mathcal{D}}{=} 
\left( \sum_{i=0}^K + \sum_{i=K+1}^\infty \right) 
\Theta_{i-1} \circ B_i =: \widetilde X_K + \widetilde X^K.
\]

For any $i > j \geq 0$
\begin{equation} \label{eq:joint-bound}
\begin{split}
& \p ( \Theta_{i-1} \circ B_i > x, \Theta_{j-1} \circ B_j > x) \\
& \leq \p ( \Pi_{i-1}  B_i > ( 1 - \varepsilon) x , 
\Pi_{j-1}  B_j > ( 1 - \varepsilon) x) \\
& \quad + 
\p ( \Theta_{i-1} \circ B_i - \Pi_{i-1} B_i  > \varepsilon x)
+ \p ( \Theta_{j-1} \circ B_j - \Pi_{j-1} B_j  > \varepsilon x)
\\
& = o( \p ( B > x)), \quad x \to \infty,
\end{split}
\end{equation}
where in the last asymptotics we used 
\eqref{eq:breiman-aux1} 
for the last two terms and for the first term (2.4.52) in \cite{BDM}.
Thus by \cite[Lemma B.6.1]{BDM} and \eqref{eq:cor}
\begin{equation} \label{eq:Xtail_K}
\p ( \widetilde X_K > x) \sim 
\p ( B > x) \, \sum_{i=0}^K (\E ( m(\xi)^\kappa))^i.
\end{equation}
The result follows from \eqref{eq:Xtail_K} with $K \to \infty$, 
provided we show that
\begin{equation} \label{eq:Xtail_inf}
\lim_{K \to \infty}
\limsup_{x \to \infty} \frac{\p ( \widetilde X^K > x) }{\p (B > x)} = 0.
\end{equation}
We have 
\begin{equation} \label{eq:tailK}
\begin{split}
\p ( \widetilde X^K > x) & \leq 
\p \left( \sum_{i=K+1}^\infty 
(\Theta_{i-1} \circ B_i - \Pi_{i-1} B_i ) > \frac{x}{2} \right) \\
& \quad + 
\p \left( \sum_{i=K+1}^\infty \Pi_{i-1} B_i > \frac{x}{2} \right).
\end{split}
\end{equation}
For the second term above, by the proof in the 
usual Grincevi\v{c}ius--Grey theorem, see \cite[(2.4.53)]{BDM}
\[
\lim_{K \to \infty} \limsup_{x \to \infty} 
\frac{\p \left( \sum_{i=K+1}^\infty \Pi_{i-1} B_i > \frac{x}{2} \right)}
{\p( B > x)} = 0.
\]
For the first term in \eqref{eq:tailK}, by Markov's inequality
with $\alpha > \kappa$ chosen as in the proof of Lemma \ref{lemma:rsum-tail}
\[
\begin{split}
& 
\p \left( \sum_{i=K+1}^\infty 
(\Theta_{i-1} \circ B_i - \Pi_{i-1} B_i ) > \frac{x}{2} \right) 
\\ & 
\leq \frac{2^\alpha}{x^{\alpha} }
\E \left| \sum_{i=K+1}^\infty 
(\Theta_{i-1} \circ B_i - \Pi_{i-1} B_i) \right|^\alpha.
\end{split}
\]

For what follows we need that if $\E (m(\xi)^\alpha) < 1$, then
for some $\rho \in (0,1)$, $c > 0$,
\begin{equation} \label{eq:Etheta-bound}
\E [ (\Theta_{n-1} \circ 1)^\alpha ] \leq c \rho^{n}.
\end{equation}
If $\alpha \leq 1$ or $\alpha > 1$ and $\E (m(\xi)^\alpha) > \E( m(\xi))$,
then Lemma 3.1 by Buraczewski and Dyszewski \cite{BD} gives precise
asymptotics. Furthermore, 
Proposition 3.1  in the first arXiv version of \cite{BD} states the 
necessary bound. For completeness, we sketch the proof as given 
in the first arXiv version of \cite{BD}. If $\alpha \leq 1$, then 
\eqref{eq:Etheta-bound} with $c = 1$ and $\rho = \E ( m(\xi)^\alpha) < 1$
follows from the conditional Jensen inequality. For $\alpha > 1$ choose 
$\varepsilon > 0$ small enough so that 
$\widetilde \rho := (1 + \varepsilon) \E (m(\xi)^\alpha))  < 1$. There is 
$C = C(\alpha, \varepsilon) > 0$ such that 
$(x + y)^\alpha \leq (1 + \varepsilon) x^\alpha + C y^\alpha$ for $x >0, y > 0$.
Therefore, writing $Z_{n} = \Theta_{n-1} \circ 1$ 
(thus $Z_n$ is a branching process in a random environment without 
immigration) we have by \eqref{eq:rsum-moment1}
\begin{equation} \label{eq:BD-bound}
\begin{split}
\E ( Z_n^\alpha ) & = 
\E \left[ \bigg( Z_{n-1} m(\xi_{n-1}) + \sum_{i=1}^{Z_{n-1}} (A_i^{(n-1)} - 
m(\xi_{n-1})) \bigg)^\alpha \right] \\
& \leq (1 + \varepsilon) \, \E ( m(\xi)^\alpha) \, \E (Z_{n-1}^\alpha)
+ c \, \E \Big( Z_{n-1}^{1 \vee \frac{\alpha}{2}} \Big) \, \E (A^\alpha) \\
& \leq \widetilde \rho \, \E ( Z_{n-1}^\alpha) + 
c \, \E (Z_{n-1}^{1 \vee \frac{\alpha}{2}}).
\end{split}
\end{equation}
For $\alpha \leq 2$, we obtain
\[
\E ( Z_n^\alpha ) \leq \widetilde \rho \, \E ( Z_{n-1}^\alpha) + 
c \,  (\E(m(\xi)))^n,
\]
which after iteration implies \eqref{eq:Etheta-bound} with 
$\rho > \max \{ \widetilde \rho, \E (m(\xi)) \}$.
Once we have the exponential decrease for $\alpha \in (1,2]$, by \eqref{eq:BD-bound}
we have it for $\alpha \in (2,4]$, and similarly \eqref{eq:Etheta-bound} 
follows for any $\alpha$ by induction.
Here, we used that the function $\lambda(\alpha)= \E (m(\xi)^\alpha)$
is convex, $\lambda(0) = 1$, and $\lambda'(0) < 0$ (by subcriticality).
\smallskip

For $\kappa > 1$ by \eqref{eq:rsum-moment1}, recalling also that 
$\tfrac{\alpha}{2} < \kappa$, we obtain
\[
\begin{split}
\E \left| \Theta_{i-1} \circ B_i - \Pi_{i-1} B_i \right|^\alpha 
& \leq c \, \E ( B^{1 \vee \frac{\alpha}{2}}) \, 
\E [ (\Theta_{i-1} \circ 1)^\alpha) ] \\
& \leq c \, \E ( B^{1 \vee \frac{\alpha}{2}}) \rho^i,
\end{split}
\]
for some $\rho < 1$, by \eqref{eq:Etheta-bound}. Therefore, using the
Minkowski inequality
\[
\begin{split}
\E \left( \left| \sum_{i=K+1}^\infty 
(\Theta_{i-1} \circ B_i - \Pi_{i-1} B_i) \right|^\alpha \right) \leq 
\left( c [ \E ( B^{1 \vee \frac{\alpha}{2}}) ]^{\frac{1}{\alpha}} \,
\sum_{i=K+1}^\infty \rho^{\frac{i}{\alpha}} \right)^\alpha \leq 
c \rho^{K}.
\end{split}
\]
For $\kappa \leq 1$ we use the Minkowski inequality for $\alpha > 1$, and 
subadditivity for $\alpha \leq 1$, 
together with \eqref{eq:rsum-moment2} and \eqref{eq:Etheta-bound},
and a similar bound follows. Therefore, \eqref{eq:Xtail_inf} holds, and thus 
the proof is complete.
\end{proof}

\begin{remark*}
The proof works in the deterministic environment case, however then 
it is much simpler. Indeed, $\Pi_{i-1} = \mu^{i}$, where $\mu = \E (A)$, and 
\eqref{eq:joint-bound} becomes trivial because of the independence of 
$B_i$ and $B_j$. Furthermore, for the second term in \eqref{eq:tailK}
instead of referring to the proof of the Grincevi\v{c}ius--Grey theorem in \cite{BDM}, 
by the Potter bounds we have for $\gamma \in (\mu, 1)$, 
$\varepsilon = \tfrac{\kappa}{2} \wedge 1$, for $x$ large enough
\[
\begin{split}
\p \left( \sum_{i= K+1}^\infty \mu^{i} B_i > \frac{x}{2} \right) 
& \leq 
\sum_{j=0}^\infty \p \left( \mu^{j+ K+1} B_j > 
\frac{x}{2} (1-\gamma) \gamma^j \right) \\
& \leq 2 \sum_{j=0}^\infty \p ( B > x ) \, 
\left( \frac{(1 - \gamma)\gamma^j}{2 \mu^{j+K+1}}  \right)^{-\kappa + \varepsilon},
\end{split}
\]
implying the necessary bound.
The first term in  \eqref{eq:tailK} can be handled the same way, since
\eqref{eq:Etheta-bound} holds for $\alpha \geq 1$ by bound (11) in 
Kevei and Wiandt \cite{KW}, while for $\alpha < 1$ it follows from 
Jensen's inequality.
\end{remark*}

\bigskip

\noindent \textbf{Acknowledgements.}
I am grateful to M\'aty\'as Barczy for useful comments and suggestions.
This research was supported by the
J\'{a}nos Bolyai Research Scholarship of the Hungarian Academy of Sciences.

\end{document}